\documentclass[11pt,a4paper]{amsart} 
\usepackage[centering,left=4.0cm,right=4.0cm]{geometry}
\usepackage{url}
\usepackage{mathrsfs}
\usepackage{enumitem}
\usepackage{color} 
\usepackage[colorlinks,linkcolor=blue,anchorcolor=blue,citecolor=blue,backref=page]{hyperref}
\usepackage{graphics,epsfig}
\usepackage{graphicx} 
\usepackage{float}
\usepackage{epstopdf}
\usepackage[utf8]{inputenc}
\usepackage{hyperref} 
\usepackage{calc, amscd}
\usepackage[T1]{fontenc} 
\usepackage{pdfpages}
\usepackage{mathtools} 
\usepackage{scalerel,stackengine}
\usepackage{caption}
\usepackage{mathptmx}
\usepackage[per-mode=symbol]{siunitx}
\usepackage{bm}
\usepackage{amsmath}
\hypersetup{breaklinks=true}
\usepackage{amsthm,mathrsfs,multirow,xcolor,lscape,longtable,pbox,lipsum}
\usepackage{amsfonts,amscd,amssymb}
\usepackage[thinlines]{easytable}
\usepackage{microtype}
\usepackage{todonotes}
\usepackage{phaistos}
\usepackage{tabularx}
\newtheorem{theorem}{Theorem}

\newtheorem{lemma}{Lemma}
\newtheorem{corollary}{Corollary}

\newtheorem{definition}{Definition}
\newtheorem{proposition}{Proposition}
\newtheorem{remark}{Remark} 
\newtheorem{question}{Question}

\newtheorem*{theorem*}{Theorem}
\newtheorem*{question*}{Question}

\usepackage{mathtools}

\newcommand{\rZ}{{\mathbb Z}}

\newcommand{\el}[1]{{\mathbf e}_{\ell}\left({#1}\right)}

\newcommand{\en}[1]{{\mathbf e}_{m}\left({#1}\right)}

\newcommand{\mr}[1]{\mathrm{#1}}

\begin{document}
\date{\today}
\author{Subham Bhakta, S. Krishnamoorthy and R. Muneeswaran  }
\address{Mathematisches Institut, Georg-August-Universit\"at G\"ottingen, Germany.} \email{subham.bhakta@mathematik.uni-goettingen.de}
\address{Indian Institute of Science Education and Research, Thiruvananthapuram, India} 
\email{muneeswaran20@iisertvm.ac.in}
\address{Indian Institute of Science Education and Research, Thiruvananthapuram, India} 
\email{srilakshmi@iisertvm.ac.in}

\newcommand\G{\mathbb{G}}
\newcommand\T{\mathbb{T}}
\newcommand\sO{\mathcal{O}}
\newcommand\sE{{\mathcal{E}}}
\newcommand\tE{{\mathbb{E}}}
\newcommand\sF{{\mathcal{F}}}
\newcommand\sG{{\mathcal{G}}}
\newcommand\sH{{\mathcal{H}}}
\newcommand\sN{{\mathcal{N}}}
\newcommand\GL{{\mathrm{GL}}}
\newcommand\HH{{\mathrm H}}
\newcommand\mM{{\mathrm M}}
\newcommand\fS{\mathfrak{S}}
\newcommand\fP{\mathfrak{P}}
\newcommand\fQ{\mathfrak{Q}}
\newcommand\Qbar{{\bar{\Q}}}
\newcommand\sQ{{\mathcal{Q}}}
\newcommand\sP{{\mathbb{P}}}
\newcommand{\Q}{\mathbb{Q}}
\newcommand{\tH}{\mathbb{H}}
\newcommand{\Z}{\mathbb{Z}}
\newcommand{\R}{\mathbb{R}}

\newcommand\Gal{{\mathrm {Gal}}}
\newcommand\SL{{\mathrm {SL}}}
\newcommand\Hom{{\mathrm {Hom}}}

\newtheorem{thm}{thm}[section]
\newtheorem{ack}[thm]{Acknowledgement}
\newtheorem{cor}[thm]{Corollary}
\newtheorem{conj}[thm]{Conjecture}
\newtheorem{prop}[thm]{Proposition}

\theoremstyle{definition}

\newtheorem{claim}[thm]{Claim}

\theoremstyle{remark}

\newtheorem*{fact}{Fact}
\newcommand{\SK}[1]{{\color{red} \sf $\heartsuit\heartsuit$ Srilakshmi: [#1]}}
\newcommand{\RM}[1]{{\color{blue} \sf $\heartsuit\heartsuit$ Muneeswaran: [#1]}}
\newcommand{\SB}[1]{{\color{purple} \sf $\heartsuit\heartsuit$ Subham: [#1]}}
%

\title{Congruence classes for modular forms over small sets}

\maketitle


\begin{abstract}
J.P. Serre showed that for any integer $m,~a(n)\equiv 0 \pmod m$ for almost all $n,$ where $a(n)$ is the $n^{\text{th}}$ Fourier coefficient of any modular form with rational coefficients. In this article, we consider a certain class of cuspforms and study $\#\{a(n) \pmod m\}_{n\leq x}$ over the set of integers with $O(1)$ many prime factors. Moreover, we show that any residue class $a\in \Z/m\Z$ can be written as the sum of at most thirteen Fourier coefficients, which are polynomially bounded as a function of $m.$
\end{abstract}



\section{Introduction}
Let $f(z)$ be a modular form of weight $k \in 2\rZ$
and level $N$ with Fourier expansion 
\begin{equation*}\label{eq:mf} 
f(z)=\sum_{n=1}^{\infty}a(n) e^{2\pi i n z}, \quad \Im (z) \ge 0,
\end{equation*}
where $a(n)$ is the $n^{th}$ Fourier coefficient. In this article, we assume that all the Fourier coefficients $a(n)\in\mathbb{Q}.$ J.P. Serre showed in \cite{Serre74} that, 
if there is a prime $p\nmid 2N$ such that $a(p) \not\equiv0 \pmod{m},$ we then have an asymptotic expansion
\begin{equation*}\label{eqn:serre1}
\#\{ n\leq x \mid a(n)\not\equiv 0 \pmod m\}=c\frac{x}{\log^{\alpha(m)} x},
\end{equation*}
for some $c>0,~0<\alpha(m)< 1.$ On the other hand, he also showed that if there exists a prime $p$ such that $a(p) \equiv 0 \pmod{m}$ for all $p\nmid 2N,$ there exists $c'>0$ such that
\begin{equation*}\label{eqn:serre2}
\#\{ n\leq x \mid a(n)\not\equiv 0 \pmod m\}=c'\sqrt{x}.
\end{equation*}
These results essentially follow from Delange's generalization of Wiener-Ikehara Tauberian theorem \cite{Delange54}, and considering the corresponding \textit{L}-function over the ring $\mathbb{Z}/m\mathbb{Z}.$ In particular, this says that $a(n)\equiv0 \pmod m$ for almost all $n$ and other residue classes appear rather rarely. However, we do not know whether each non-zero residue class $a\in \Z/m\Z$ can be written as $a(n)\pmod m$ with equal proportion. It was mentioned by Serre in page 20 of \cite{Serre74} that, for any odd $m$ and non-zero $a\in \Z/m\Z,$
\begin{equation*}\label{eqn:any}
\#\{ n\leq x \mid a(n)\equiv a \pmod m, \  \omega(n)=M\} \gg \frac{x}{\log x}(\log \log x)^M.
\end{equation*}
Proof of these arguments were based on showing that, for a positive density of primes $p\equiv 1 \pmod{m}$ and $q\equiv -1 \pmod {m}$ the corresponding Hecke operators $T_p$ and $T_q$ acts respectively as $2$ and $0$ on the $\Z$-module of all holomorphic modular forms with coefficients in $\Q.$

In this regard, we study the distribution of  $\{a(n) \pmod m\}$ over the integers with $O(1)$ prime factors and prove the following.

\begin{theorem}\label{intro:thm14}
Let $M\geq 1$ be any integer, and $f$ be any newform without CM, and with coefficients in $\Q.$ Then under the certain assumptions on $m$, the following asymptotic formula holds for any tuple 
    $$\frac{\#\{n \leq x \mid a(n)\equiv a\pmod{m},~\omega(n)=M\}}{\#\{n \leq x \mid \omega(n)=M\}}\sim d_{a}(m)\frac{1}{m^{M}},$$
    for some $d_{a}(m)>0$, which is an effectively computable constant.
   \end{theorem}
We shall write a more precise version in Theorem~\ref{thm:thm0}, and for a much broader family of cuspforms. Even though this guarantees the existence of a solution to $a(n)\equiv a \pmod m,$ any such $n$ could be large. First of all, a solution to $a(n)\pmod m$ is also a solution to $a(n)\pmod{\ell}$ for any prime $\ell \mid m.$ The reader may refer to \cite{Thorner2020} for a bound on least such $n$ when $m$ is a prime. Thorner and Zaman assumed that $f(z)$ is a newform without CM, when weight $k$ satisfies $(k-1,m-1)=1,$ and the associated Galois representation has image $\text{GL}_2(\Z/m\Z).$ Thorner and Zaman gave a bound for least prime $n$ which satisfies $a(n)\equiv a\pmod m.$ However, this is indeed equivalent to the general case. For instance if $n_0$ is the least solution, then writing $n_0=\prod_{i=1}^{t} p^{e_i}_{i,0},$ we have $\prod_{i=1}^{t} a(p^{e_i}_{i,0})\equiv a\pmod m.$ Setting $a(p^{e_i}_{i,0})=a_i$ and writing $a(p^{e_i}_{i,0})$ as a polynomial in $a(p_{i,0})$ and (since $p_{i,0}$ is fixed), we see that $a(p_{i,0})\pmod \ell, \ \ell\mid m$ is a root of a polynomial over $\mathbb{F}_{\ell}[x].$ Then we look for the least prime $p_i$ for which $a(p_i)\equiv a(p_{i,0})\pmod{\ell},$ and this shows that it is enough to look for the least prime.

In particular, we find that any such $n$ with $a(n)\equiv a\pmod m$ could grow exponentially as a function of $m.$ However, it turns out that we can have a polynomial bound if we allow a few more terms. More precisely, we study the equation 
$$\sum_{i=1}^{O(1)} a(n_i)\equiv a \pmod m,~n_i=m^{O(1)}.$$
In~\cite{Shparlinski2005}, Shparlinski proved that for the Ramanujan-Tau function $\tau$ the set $\{\tau(n)\}_{n\geq 1}$ is an additive basis modulo any prime $\ell,$ that is, there exists an absolute constant $s$ such that the Waring-type congruence
\[\tau(n_1)+\cdots+ \tau(n_s) \equiv a \hspace{-0.3cm}\pmod{\ell}
  \]    
is solvable for any residue class $a \pmod{\ell}.$ Moreover, Shparlinski showed that each such $n_i$ is at most $O(\ell^4),$ and has at most $4$ prime factors. His argument relied on studying the exponential sum 
$$\sum\limits_{\substack{u_1\in U_1,u_2\in U_2\\a(n)=u_1u_2 \pmod \ell}}\el{\lambda(u-a)},~\forall \lambda\in \mathbb{F}_{\ell}^{*},$$
over a suitably large subset $U_1\times U_2$ of $\mathbb{F}_{\ell}\times \mathbb{F}_{\ell},$ which satisfies $\#U_1 \# U_2>\ell^{1+c}$ for some $c>0.$  

Recently, Bajpai, Garc\'{\i}a, and the first author showed in \cite{BBG22} that when $f$ is a newform without CM and with rational Fourier coefficients,  then there is an absolute constant $s_0$ such that any element of $\mathbb{F}_{\ell}$ can be written as a sum of at most $s_0$ elements of the set $\{a(p^n)\}_{n\geq 0},$ for almost all primes $p$ and $\ell$. Moreover, $s_0$ does not depend on the choice of $f.$ To prove that they studied the exponential sum 
 \begin{equation*}\label{eq:main1}
    \sum_{n\le \Gamma}\el{\lambda a(p^n)},~\forall \lambda\in \mathbb{F}_{\ell}^{*},
  \end{equation*}
where $\Gamma$ is a suitable parameter depending on the sequence $\{a(p^n) \pmod{\ell}\}.$ Although we are getting more solutions, in this case, the coefficients might grow exponentially as a function of $\ell.$ 

In this article, we prove an improvement over the main result of Shparlinski~\cite{Shparlinski2005} in the following form.

\begin{theorem}\label{thm:main}
Let $f(z)$ be any cuspform with rational coefficients. Let $0<\epsilon,\gamma<1$ be any given positive real numbers. Let $S_1, S_2$ be any set of primes having positive density with $S_1\cap S_2=\phi.$ Then there exists an integer $N_{S_1,S_2}$ such that for any integer $m$ with all prime factors larger than $N_{S_1,S_2},$ and $L^{\gamma}\geq m/L,$ where $L$ is the largest prime factor of $m,$ we have the following.
\begin{itemize}
    \item [(a)] If $f(z)$ is a Hecke eigenform then for any $a\in \Z/m\Z,$ we can write
$$\sum_{i=1}^{s} a(n_i) \equiv  a\pmod{m},~n_i\leq m^{4\epsilon},~\forall 1\leq i\leq s,$$ 
for some $s=O_{\varepsilon,\gamma}(1)$. Moreover, all the prime factors of any such $n_i$ are $O(m^{\epsilon}),$ and in $S_1\cup S_2.$\\
\item [(b)] In general if $f(z)$ is of the form $\sum_{i=1}^{r} c_if_i, \  c_i\in \mathbb{Q}, f_i$ are newforms without CM and $\sum \sigma_ic_i\not\equiv0\pmod{m},\sigma_i\in\{\pm 1\}.$ If none of the associated Galois representations $\rho_{f_{i_1},f_{i_2}\cdots,f_{i_s},m}$ does not have image $\Delta_{k}^{(s)}(m)$ for any subset $I=\{i_1,i_2,\cdots, i_s\}$ of $\{1,2,\cdots, r\}$ with $\#I\geq 2.$ Then for any $a\in \Z/m\Z,$ we can write
$$\sum_{i=1}^{s} a(n_i)\equiv a\pmod{m},~n_i\leq m^{4\epsilon},~\forall 1\leq i\leq s,$$
where $s$ is the same as in part (a), all the prime factors of any such $n_i$ are $O(m^{\epsilon})$, and are in $S_1\cup S_2.$
\end{itemize}
\end{theorem}
One of the key tools to prove Theorem~\ref{thm:main} is the following estimate on the bilinear exponential sums, which says 
$$\mathrm{max}_{\lambda\in \Z/m\Z}\left|\sum_{a_1\in A_1}\sum_{a_2\in A_2}\en{\lambda a_1a_2}\right|\leq \sqrt{m\#A_1\#A_2},$$
for any subsets $A_1,A_2$ of $\Z/m\Z$. We shall briefly discuss the bilinear and multilinear exponential sums in Section~\ref{sec:wring}. Applying this unconditional exponential bound above and the \textit{sum-product estimate} for subsets of finite fields, we can make everything explicit in the proof of Theorem~\ref{thm:main}, and obtain the following. 

\begin{corollary}\label{thm:main0}
Let $f(z)$ be any Hecke eigenform with rational coefficients, and $S_1, S_2$ be any set of primes having positive density with $S_1\cap S_2=\phi.$ Then there exists an integer $N_{S_1,S_2}$ such that for any integer $m$ with all prime factors larger than $N_{S_1,S_2},$ and $L^{1/77}\geq m/L,$ where $L$ is the largest prime factor of $m,$ and for any $a\in \Z/m\Z,$ we can write
$$\sum_{i=1}^{s} a(n_i) \equiv  a\pmod{m},~n_i\leq m^{130/33},~\forall 1\leq i\leq 13,$$
for some $s\leq 52$. Moreover, all the prime factors of any such $n_i$ are $O(m^{65/66}),$ and in $S_1\cup S_2.$
\end{corollary}

Now, it is natural to ask whether obtaining solutions with smaller prime factors is possible. In Theorem~\ref{thm:main}, we assume that all the prime factors of $m$ are large enough and that the largest prime factor of $m$ is larger than the other prime factors. Of course, any prime $m$ satisfies these criteria. To obtain solutions for composite with small prime factors, we assume that $m$ is squarefree. In this regard, we prove the following.

\begin{theorem}\label{thm:main2}
Let $f(z)$ be any cuspform with rational coefficients, $0<\varepsilon,\gamma<1$ be any given real numbers, $m$ be a square-free positive integer and $S_1, S_2,\cdots, S_{\omega}$ be any set of primes of positive density, with $\varepsilon(2+\frac{\omega-2}{81})>2(\gamma+1)$
and $S_i\cap S_j=\phi , \ i\neq j.$ Then there exists an integer $N_{S_1,S_2,\cdots,S_{\omega},\varepsilon}$ such that for any integer $m$ with all the prime factors of $m$ are larger than $N_{S_1,S_2,\cdots,S_{\omega},\varepsilon},~m^{\varepsilon/2}=o(L)$ and $L^{\gamma}\geq m/L$ for some $\gamma>0$, we have the following.
\begin{itemize}
    \item [(a)] If $f(z)$ is a Hecke eigenform then for any $a\in \Z/m\Z,$ we can write
$$\sum_{i=1}^{s} a(n_i)\equiv a \pmod{m},~n_i\leq m^{2\varepsilon\omega},~\forall 1\leq i\leq s,$$ 
for some computable $s$ depending on $\varepsilon,\omega,\gamma.$ Moreover, all the prime factors of any such $n_i$ are less than or equal to $m^{\varepsilon}$ and in $\bigcup\limits_{i=1}\limits^{\omega}S_i.$\\
\item [(b)] In general if $f(z)$ of the form $\sum\limits_{i=1}\limits^{r} c_if_i, c_i\in \mathbb{Q}, \ f_i$ are newforms without CM and $\sum \sigma_ic_i\not\equiv0\pmod{m},\sigma_i\in\{\pm 1\}.$ If the associated Galois representation $\rho_{f_1,f_2\cdots,f_r,m}$ does not have image $\Delta_{k}^{(r)}(m),$ then for any $a\in \Z/m\Z,$ we can write
$$\sum_{i=1}^{s} a(n_i)\equiv a\pmod{\ell},~n_i\leq m^{2\varepsilon\omega},~\forall 1\leq i\leq s,$$
for the same $s$ as in $(a).$ Moreover, all the prime factors of any such $n_i$ are less than or equal to $m^{\varepsilon}$ and in $\bigcup\limits_{i=1}\limits^{\omega}S_i.$
\end{itemize}
\end{theorem}
In part $(b)$ of both of the theorems above, we are working with a wider class of cuspforms that need not be Hecke eigenforms. We know that any cuspform $f(z)$ can be uniquely written as a $\mathbb{C}$-linear combination of pairwise orthogonal Hecke eigenforms with Fourier coefficients coming from $\mathbb{C}$. See~\cite[Chapter 5]{DS2005} for a brief review of the Hecke theory of modular forms. We are concerned with all such cuspforms, which can be uniquely written as a $\mathbb{Q}$-linear combination of pairwise orthogonal eigenforms with Fourier coefficients coming from $\mathbb{Q}.$ The reader may note that not every cuspform of weight $k$ with rational coefficients has this property.\footnote{For one example, the reader may look at the answer of Jeremy Rouse in https://mathoverflow.net/questions/364763/cusp-forms-with-integer-fourier-coefficients.}

\subsection{Notations} 
By $O_{S_1, S_2,\cdots, S_k}(B)$ we mean a quantity with absolute value at most $cB$ for some positive constant positive constant $c$ depending on $S_1,\cdots, S_K$ only; if the subscripts are omitted, the implied constant is
absolute. We write $A\ll_{S_1, S_2,\cdots, S_k} B$ for $A=O_{S_1, S_2,\cdots, S_k}(B)$, and $A=o(B)$ for $A/B\to 0$.

For any integer $n,$ we denote $\omega(n)$ to be the number of distinct prime factors of $n,$ and $\mu(n)$ to be the M\"obius function. We shall also denote $(\Z/m\Z)^{*}$ to be the multiplicative group modulo $m$ and $((\Z/m\Z)^{*})^{k}$ to be $k^{\mathrm{th}}$ power of elements in $(\Z/m\Z)^{*}.$

\section{Distribution over small sets}\label{sec:basicgalois}
In this section, we will discuss the preliminary facts needed to prove the main results.

\subsection{Galois representations and image}\label{sec:galois}
Let $f(z)$ be any newform of weight $k$ and level $N$. From Deligne-Serre correspondence, we have an associated Galois representation
\[\rho_{f}^{(\ell)}:\text{Gal}\left(\overline{\mathbb{Q}}/\mathbb{Q}\right) \longrightarrow \mathrm{GL}_2\left(\mathbb{F}_{\ell}\right),\]
such that $a(p)\pmod{\ell}\equiv \text{tr}\left(\rho_{f}^{(\ell)}(\text{Frob}_p)\right)$ for any prime $p \nmid N\ell.$ It is clear that the characteristic polynomial of $\rho_f^{(\ell)}(\text{Frob}_p)\pmod{\ell}$ is same as $x^2-a(p)x+p^{k-1}\hspace{-0.1cm}\pmod{\ell}$. When $f$ is without CM it follows from Ribet~\cite[Theorem 3.1]{Ribet85} that, the image of this representation is given by $$\mathrm{Im}\left(\rho_f^{\ell}\right)=\Delta_{k}(\ell)=\left\{A \in \text{GL}_2\left(\mathbb{F}_{\ell}\right) \mid \det(A) \in (\mathbb{F}_{\ell}^{*})^{k-1}\right\},$$ except possibly for finitely many primes $\ell.$ For any integer $e\geq 1,$ we have a Galois representation 
\[\rho_{f,\ell^{e}}:\text{Gal}\left(\overline{\mathbb{Q}}/\mathbb{Q}\right) \longrightarrow \mathrm{GL}_2\left(\mathbb{F}_{\ell^e}\right)\]
satisfying that $a(p) \pmod {\ell^e}\equiv \text{tr}\left(\rho_{f,\ell^e}(\text{Frob}_p)\right),$ and 
$$\mathrm{Im}(\rho_{f,\ell^{e}})=\Delta_{k}(\ell^e):=\{A\in \mathrm{GL}_2\left(\mathbb{F}_{\ell^e}\right) \mid \det(A)\in (\mathbb{F}_{\ell^e}^{*})^{k-1}\}.$$

Now, given any composite number $m$, we can naturally associate a Galois representation 
$\rho_{f,m}:=\prod_{\ell^e\mid\mid m}\rho_{f,\ell^e},$
and denote $\mr{G}_{f,m}$ to be its image. Due to Ribet's result, we immediately have the following.
\begin{corollary}\label{lem:compimage}
Suppose that $f(z)$ is any newform without CM. Then, there exists a finite set of primes $S_f$ such that 
$$\mr{G}_m=\Delta_k(m)=\{A\in \mathrm{GL}_2\left(\Z/m\Z\right) \mid \det(A)\in ((\Z/m\Z)^{*})^{k-1}\},$$ 
for any integer $m$ co-prime to any prime from $S_f.$ 
\end{corollary} 

Consider the prime factorization $m=\prod_{\ell\mid m}\ell^e.$ Take any residue class $a\in \mathbb{Z}/\ell^e\mathbb{Z}$ and $\lambda \in (\Z/\ell^e\Z)^{*}.$ Denote $N_{a,\lambda}(\ell^e)$ be the number of matrices in $\Delta_k(\ell^e)$ of trace $a$ and determinant $\lambda.$ We know from \cite[Corollary 6.0.7]{mathewson2012class} that 
\begin{equation}\label{eqn:factors}
N_{a,\lambda}(\ell^e)=\begin{cases}
\ell^{2e},~\mathrm{if}~a^2-4\lambda=0.\\
\ell^e(\ell^e + 1),~\mathrm{if}~a^2-4\lambda~\mathrm{is~a~non}\text{-}\mathrm{zero~square~in~\Z/\ell^e\Z}. \\
\ell^e(\ell^e-1)~\mathrm{if}~a^2-4\lambda~\mathrm{is~not~a~square~in~\Z/\ell^e\Z}.
\end{cases}
\end{equation}

\begin{lemma}\label{lem:existence}
Let $m=\prod_{\ell\mid m}\ell^e$ be any integer, and $a$ be any residue class in $\Z/m\Z,$ and $\lambda\in (\Z/m\Z)^{*}.$ Then we have, 
$$N_{a,\lambda}(m):=\#\{A\in \Delta_k(m) \mid \mathrm{tr}(A)= a,~\det(A)=\lambda\}=m^2+O\left(\frac{2^{\omega(m)}m^2}{\mathcal{L}}\right),
$$
where $\mathcal{L}$ is the smallest prime factor of $m.$
\end{lemma}
\begin{proof}
By the Chinese remainder theorem, we can write
$$\Delta_k(m)=\prod_{\ell^e \mid\mid m}\Delta_k(\ell^e).$$
Therefore, it is enough to prove the result when $m$ is a prime power. The result now follows from (\ref{eqn:factors}). 
\end{proof}
\begin{remark}\label{rem:equi}\rm
In particular, for any $m$ with large enough prime factors, trace values of the matrices in $\Delta_k(m)$ are equidistributed in $\Z/m\Z$. We talk about equidistribution in terms of the \textit{density} measure on $\Z/m\Z$, which is defined by, 
$$A\mapsto \frac{\# A}{\Z/m\Z}=\frac{\# A}{m},~\forall A\subseteq \Z/m\Z.$$ 
\end{remark}

Now let $f_1,f_2,\cdots, f_r$ be a set of newforms, of weights respectively $k_1,k_2,\cdots,k_r.$ Then we can associate a Galois representation $\rho_{f_1,f_2,\cdots,f_r,m}:\mathrm{Gal}\left(\overline{\mathbb{Q}}/\mathbb{Q}\right) \to \mathrm{GL}_{2r}\left(\Z/m\Z\right)$ defined by the map
\begin{equation}\label{eqn:map}
\sigma \mapsto \begin{psmallmatrix}
    \rho_{f_{1,\ell}}(\sigma) & & \\
    & \rho_{f_{2, \ell}}(\sigma) & &\\
    & & \ddots &\\
    & & & \rho_{f_{r,\ell}}(\sigma)\\ 
  \end{psmallmatrix}.
  \end{equation}
Such that the image $G_{f_1,f_2,...,f_r,m}$ is contained in $\Delta_{k_1,k_2,\cdots,k_r}(m),$ where $\Delta_{k_1,k_2,\cdots,k_r}(m)$ denotes the set of all block matrices of size $2\times 2$ in $\mathrm{GL}_{2}\left(\Z/m\Z\right)$ in which determinant of each block is a $k_i-1^{th}$ power of some element in the multiplicative group $(\mathbb{\Z}/{m\Z})^{*}.$

\begin{definition}
Let $\ell$ be a prime. Two newforms $f_i$ and $f_j$ of weight and level respectively $k_i,k_j$ and $N_i,N_j,$ are said to be $\ell$-equivalent, i.e. $f_i\sim_{\ell} f_j,$ if there exists a quadratic character 
$$\chi:(\Z/N_iN_j\ell\Z)^{*}\to \mathbb{C}^{*},$$
satisfying $a_{f}(p)=\chi(p)a_{g}(p)\pmod{\ell}$ for any prime $p\nmid N_1N_2.$ Moreover we say that $f_i$ and $f_j$ are twist equivalent, i.e. $f_i\sim f_j,$ if there exists a quadratic character $\chi$ satisfying $a_{f}(p)=\chi(p)a_{g}(p)$ for any prime $p\nmid N_iN_j.$
\end{definition}
The reader may note that if $f_i\sim f_j$, then their weights $k_i$ and $k_j$ should be the same. More importantly, we have the following.
\begin{lemma}\label{lem:equiv}
If two newforms $f_i,f_j$ are not twist-equivalent, then they are $\ell$-equivalent for only finitely many primes $\ell.$
\end{lemma}
\begin{proof}
For the sake of contradiction, let us assume that $f_i\sim_{\ell}f_j$ for infinitely many primes $\ell.$ Then for each prime $p\nmid N_iN_j$ there exists infinitely many primes $\ell>N_1N_2p$ satisfying $a_i(p)\equiv \pm a_j(p) \pmod{\ell}.$ It is evident that for one of the sign $\sigma_p \in \{\pm 1\},$ we have $a_p(p)=\sigma_pa_j(p).$ Now define a quadratic Dirichlet character $\chi$ modulo $N_iN_j$ with $\chi(p):=\sigma_p.$ 
\end{proof}
We now recall the main result of Masser and W\"{u}stholz in \cite{MW}, where they worked with the image of the product of the elliptic curves. In other words, when all the $f_i$ in (\ref{eqn:map}) has weight $2$, and $m$ is a prime. We have the following generalization in the case of composite modulus.
\begin{lemma}\label{lem:notsurj}
Let $f_1,f_2,\cdots, f_r$ by any set of pairwise twist-inequivalent newforms without CM. Then there exists a finite set of primes $S_f$ such that, $\mr{G}_{f_1,f_2,\cdots,f_r,m}$ contains $\mathrm{SL}_2(\Z/m\Z)^r$ for any integer $m$ co-prime to any prime from $S_f.$   
\end{lemma}
Before proving this, let $\mr{G}$ be any finite group and denote $\text{Occ}(\mr{G})$ be the isomorphism classes of non-abelian simple groups, which are the quotients of composite factors of $\mr{G}.$

 \begin{proof}
It follows from Lemma~\ref{lem:equiv} that, there exists a finite set of primes $P_f$ such that, any of $f_i,f_j$ are not $\ell$-equivalent for any prime $\ell \not\in P_f.$ Then we make $P_f$ bigger if necessary, to ensure that each $\rho_{f_i,\ell}$ has image $\Delta_{k_i}(\ell),$ and in particular contains $\text{SL}_2(\Z/\ell\Z).$ Now it follows from Lemma 5.1 in \cite{MW} by taking $e=\ell-1$ that, image of $\mr{G}_{f_i,f_j,\ell}$ contains $\text{SL}_2(\mathbb{Z}/\mathbb{Z}\ell)^2$ for any $\ell \not\in P_f.$ Now it follows from Lemma 5.2.2 of \cite{Ribet76} that $\mr{G}_{f_1,f_2,\cdots,\ell}$ contains $\text{SL}_2(\mathbb{Z}/\ell \mathbb{\Z})^r,$ since $\text{SL}_2(\mathbb{Z}/\ell \mathbb{\Z})$ is a self commutator for any prime $\ell\geq 5,$ 
 the image contains any matrix of type $(I,I,\cdots, \text{SL}_2(\mathbb{Z}/\ell \mathbb{\Z}),I,\cdots, I).$ 

Moreover, any $(I,I,\cdots, \text{PSL}_2(\mathbb{Z}/\ell \mathbb{\Z}),I,\cdots, I)\in \text{Occ}(\mr{G}_{f_1,f_2,\cdots,f_r,m})$ for any such integer $m,$ as long as all the prime factors of $m$ are larger than $5.$ Then it follows from Theorem 2 in the appendix of \cite{Coj} that $\mr{G}_{f_1,f_2,\cdots,f_r,m}$ contains $(I,I,\cdots, \text{PSL}_2(\mathbb{Z}/m \mathbb{\Z}),I,\cdots, I)$ and this completes the proof.
\end{proof} 
Now we need one more lemma to prove the main result of this section.
\begin{lemma}\label{lem:explimage}
Let $f_1,f_2,\cdots, f_r$ by any pairwise twist-inequivalent newforms without CM of the same weight $k$. Then there exists a finite set of primes $S_f$ such that,  
$$\mr{G}_{f_1,\cdots,f_r,m}=\Delta_{k}^{(r)}(m):=\left\{\begin{psmallmatrix}
    A_1 & & \\
    & A_2 & &\\
    & & \ddots &\\
    & & & A_r\\ 
  \end{psmallmatrix} \mid \det(A_1)=\cdots=\det(A_r)\in (\Z/m\Z^{*})^{k-1}\right\}.$$ 
for any integer $m$ co-prime to any prime from $S_f.$ 
\end{lemma}
\begin{proof}
It follows from Corollary~\ref{lem:compimage}, and by induction that each projection $$\pi_i:\mr{G}_{f_1,\cdots,f_r,m}\to \text{GL}_2(\Z/m\Z)$$
has the property that $\text{Im}(\pi_i)=\Delta_k(m),~\forall 1\leq i\leq r.$ For $r=2,$ the result follows from the proof of Lemma 3.3 in \cite{Jones13}, combining with Lemma~\ref{lem:notsurj}. Note that we are using Lemma~\ref{lem:notsurj} to rule out the case (b) of Lemma 3.3 in \cite{Jones13}.

Now, by induction, 
$$\mr{G}_{f_1,\cdots,f_{r-1},m}=\left\{\begin{psmallmatrix}
    A_1 & & \\
    & A_2 & &\\
    & & \ddots &\\
    & & & A_r\\ 
  \end{psmallmatrix} \mid \det(A_1)=\cdots=\det(A_{r-1})\in (\Z/m\Z^{*})^{k-1}\right\}$$
  and $\mr{G}_{f_r,m}=\Delta_k(m).$ If $\mr{G}_{f_1,\cdots,f_{r},m}$ does not have the desired image, then by Goursat's lemma (Lemma 3.2 in \cite{Jones13}) there exists a normal subgroup $N_1$ of $\mr{G}_{f_1,\cdots,f_{r-1},m}$ and a normal subgroup $N_2$ of $\mr{G}_{f_r,m}$ and an isomorphism $\psi:\mr{G}_{f_1,\cdots,f_{r-1},m}/N_1\to \mr{G}_{f_r,m}/N_2$ such that 
  $$\mr{G}_{f_1,\cdots,f_r,m}=\{(g_1,g_2) \mid \psi(g_1N_1)=g_2N_2\}.$$
  If $N_2$ contains $\text{SL}_2(\Z/m\Z)$ then clearly $N_1$ is contained in $\text{SL}_2(\Z/m\Z)^{r-1}$, because
  $$N_1\times \text{SL}_2(\Z/m\Z) \subseteq N_1\times N_2 \subseteq \Delta_k^{(r)}(m).$$
 Due to the isomorphism $\psi$, we have $N_1=\text{SL}_2(\Z/m\Z)^{r-1}$ and $N_2=\text{SL}_2(\Z/m\Z).$ In particular, 
 $$\mr{G}_{f_1,\cdots,f_{r-1},m}/N_1= \mr{G}_{f_r,m}/N_2=(\Z/m\Z)^{*}.$$
 This implies that $\psi$ must be the identity map, because $\mr{G}_{f_1,\cdots,f_r,m}$ is the graph of $\psi,$ which of course lies inside $\Delta^{(r)}_k(m).$ The proof is now complete.
\end{proof}

With this, we are done with all the preparations to prove the main result of this section.

\subsection{Distribution of Fourier coefficients}
In this section, we study $\{a(n) \pmod m\}$ when $n$ has only finitely many prime factors. The one prime factor case is just an immediate consequence of Lemma~\ref{lem:explimage} and Chebotarev's density theorem when $a$ is given by a newform. For a larger family of cuspforms, we have the following result when we study over $n$ with $\omega(n)=1$.

\begin{proposition}\label{thm:expldist}
Let $f=c_1f_1+c_2f_2\cdots+c_rf_r$ be any cuspform with coefficients in $\Q.$ Assume that all the $f_i$ are pairwise twist-inequivalent newforms without CM of same weight $k$. Then there exists an integer $N_f$ such that for any integer $m$ co-prime to $N_f$ satisfying $m\sim \phi(m)$, and $2^{\omega(m)+r}=o(\mathcal{L}),$ where $\mathcal{L}$ is the smallest prime factor of $m$, the set
$$\{a(p) \pmod m \mid~p,~\mathrm{prime}\},$$
is equidistributed among the residue classes
modulo $m$.
\end{proposition}
\begin{proof}
 Take any residue class $a\in \Z/m\Z.$ It follows from Lemma~\ref{lem:existence} that for each tuple $(a_1,a_2,\cdots,a_r)\in (\Z/m\Z)^r$ and $(\lambda_1,\lambda_2,\cdots, \lambda_r)\in ((\Z/m\Z)^{*})^r,$
$$\#\{\mathrm{diag}(A_1,\cdots, A_r)\in (\mathrm{GL}_{2}(\Z/m\Z))^r\mid \mathrm{det}(A_i)=\lambda_i,~\mathrm{tr}(A_i)=a_i~\forall 1\leq i\leq r\}.$$
is $m^{2r}+O(2^{r+\omega(m)}m^{2r}/\ell)$, where $\ell$ is the least prime factor of $m$. In particular, 
\begin{equation}\label{eqn:counting}
\# \{A\in \Delta_k^{(r)}(m)\mid \mathrm{tr}(A_1)=a_1,\cdots ,\mathrm{tr}(A_r)=a_r \}=\frac{\phi(m)}{(\phi(m),k-1)}(m^{2r}+O(2^{r+\omega(m)}m^{2r}/\ell)).
\end{equation}
Therefore, the set 
$$\{p \mid a_1(p)=a_1,~a_2(p)=a_2,\cdots a_r(p)=a_r\}$$
has density $\frac{\frac{\phi(m)}{(\phi(m),k-1)}(m^{2r}+O(2^{r+\omega(m)}m^{2r}/\ell))}{\#\Delta^{(r)}_k(m)},$ which is precisely $\frac{m^{2r}+O(2^{r+\omega(m)}m^{2r}/\ell)}{\#\text{SL}_2(\Z/m\Z)^r}.$
Now the number of tuples $(a_1,a_2,\cdots,a_r)$ with $c_1a_1+c_2a_2\cdots+c_ra_r=a$ is $m^{r-1}.$ Therefore we have $a(p)=a,$ for a set of primes $p$ with density 
\begin{equation}\label{eqn:onecase}
m^{r-1}\frac{m^{2r}+O(2^{r+\omega(m)}m^{2r}/\ell)}{\#\text{SL}_2(\Z/m\Z)^r}\sim \frac{1}{m},
\end{equation}
for any $m$ satisfying $n\sim \phi(m)$, and $2^{\omega(m)+r}=o(\mathcal{L}),$ where $\mathcal{L}$ is the smallest prime factor of $m.$ 
    \end{proof}

Now to study $\{a(n) \pmod m\}$ with $n$ having more than one prime factor. For any integer $M\geq 1$, denote
$$N_M(x)=\{n\leq x \mid \omega(n)=M\}.$$
It is a classical result \cite{MV07} that $\#N_M(x)\sim \frac{1}{(M-1)!}\frac{x}{(\log x)} (\log \log x)^{M-1}$. We now need the following generalization. \footnote{For proof, the reader may refer to \href{https://mathoverflow.net/questions/156982/chebotarev-density-theorem-for-k-almost-primes}{https://mathoverflow.net/questions/156982/chebotarev-density-theorem-for-k-almost-primes}.} 
\begin{lemma}\label{lem:m-prime}
    Let $P_1$, $\ldots$, $P_r$ be disjoint subsets of the primes with density respectively $\alpha_1$, $\ldots$, $\alpha_r$. Let $N(M;a_1,\ldots,a_r)$ denote the set of integers that are products of $M$ primes with exactly $a_j$ of these primes chosen from the set $P_j$. Let us assume that all $\alpha_j\geq 0, ~a_j \ge 1$, then for fixed any $M$ and as $x\to \infty$
$$ 
\sum_{\substack {{n\le x}\\ {n\in N(M;a_1,\ldots,a_r)}}} 1 
\sim M\left(\prod_{j=1}^{r} \frac{\alpha_j^{a_j}}{(a_j)!} \right)\frac{x}{(\log x)} (\log \log x)^{M-1}.
$$
\end{lemma}
We now consider the factorizations of any $a\in \mathbb{Z}/m\mathbb{Z}$ in-to $M$ many terms. Let us write a factorization $a=a_1a_2\cdots a_M$. It may be possible that some of the $a_i$ repeats in this factorization. Now given any tuple $\vec{a}=(a_1,a_2,\cdots, a_M)\in (\mathbb{Z}/m\Z)^M,$ denote $p(\vec{a})=a_1a_2\cdots a_{M}.$ We say that two vectors $\vec{a_{1}}$ and $\vec{a_2}$ are equivalent, i.e. $\vec{a_{1}}\sim_{S_M}\vec{a_{2}}$ if and only if they differ by a permutation in $S_M$. Given any element $\vec{a}\in (\mathbb{Z}/m\Z)^M$, denote $n_{\vec{a}}=\prod_{1\leq i\leq k}n_i!$, where $a_1,a_2,\cdots, a_k$ are the set of all distinct terms that appear in $\vec{a}$ with $a_i$ appearing $n_i$ times. In particular, we have $\sum_{1\leq i\leq k} n_i=M$. We shall use these notations to study the case of newforms in the next theorem. 

To generalize that, we need to work with $M_{r\times M}(\Z/m\Z)$, the ring of matrices over $\Z/m\Z$ with $r$-rows and $M$-columns. Then we consider the natural action of $S_M$ on the columns of $M_{r\times M}(\Z/m\Z)$. Given any element $A\in M_{r\times M}(\Z/m\Z)$, denote $C_1(A),~C_2(A),\cdots C_M(A)$ and $R_1(A),~R_2(A),\cdots R_r(A)$ respectively be the columns, and the rows of $A$. Moreover, denote $n_{A}$ to be the number $\prod_{1\leq i\leq k}n_i!$, where $C_1,C_2,\cdots, C_k$ be the set of all distinct columns that appear in $A$ with $C_i$ appearing $n_i$ times.

\begin{theorem}\label{thm:thm0}
Let $M\geq 1$ be any integer, and $f=c_1f_1+c_2f_2\cdots+c_rf_r$ be any cuspform with coefficients in $\Q.$ Assume that all the $f_i$ are pairwise twist-inequivalent newforms without CM of same weight $k$. Then there exists an integer $N_f$ such that for any integer $m$ co-prime to $N_f$ satisfying $m\sim \phi(m)$, and $2^{\omega(m)+r}=o(\mathcal{L}),$ where $\mathcal{L}$ is the smallest prime factor of $m$, the following asymptotic formula holds for any tuple $\vec{a}=(a_1,a_2,\cdots, a_r)\in (\Z/m\Z)^r$.
    $$\frac{\#\{n\in N_M(x) \mid a_1(n)=a_1,a_2(n)=a_2\cdots, a_r(n)=a_r\}}{\#N_M(x)}\sim d_{\vec{a}}(m)\frac{1}{m^{rM}},$$
    for some $d_{\vec{a}}(m)>0$, which is an effectively computable constant.
   \end{theorem}
\begin{proof}
Let us first do the case $r=1$. We use Lemma~\ref{lem:m-prime} and Proposition~\ref{thm:expldist} to get
\begin{equation}\label{eqn:oneterm}
\frac{\#\{n\in N_M(x) \mid a(n)=a\}}{\#N_M(x)}  \sim \frac{1}{m^{M}}\sum_{\substack{\vec{a}\in (\mathbb{Z}/m\Z)^M/S_M\\ p(\vec{a})=a}}\frac{M!}{n_{\vec{a}}}.
\end{equation}
Now for the case $r\geq 1$, we use the proof of Proposition~\ref{thm:expldist} for $r\geq 1$. More precisely, for each tuple $(a^{(1)},a^{(2)},\cdots, a^{(r)})\in (\Z/m\Z)^r$, the set
$$\{p\mid a_1(p)=a^{(1)},a_2(p)=a^{(2)},\cdots, a_r(p)=a^{(r)}\}$$
has density $\sim \frac{1}{m^r}$. Similarly, as in the case $r=1$, it follows from Lemma~\ref{lem:m-prime} that the required proportion is given by
$$\frac{1}{m^{rM}}\sum_{\substack{A\in M_{r\times M}(\Z/m\Z)/S_M\\ p(R_1(A))=a_1,~p(R_2(A))=a_2,\cdots p(R_r(A))=a_r}} \frac{M!}{n_{A}},$$
the proof is now complete taking $d_{\vec{a}}(m)$ to be the summation in the above line.

\end{proof}

\begin{remark}\rm
The reader may note that the proportion that we are having in (\ref{eqn:oneterm}) is indeed less than $1$, as $\#\{\vec{a}\in (\mathbb{Z}/m\Z)^M/S_M\}\ll m^{M}/M!$, and all the $n_{\vec{a}}\geq 1$. Moreover, the condition should $p(\vec{a})=a$ should hold roughly with proportion $\frac{1}{m}$. Therefore, the proportion in \ref{eqn:oneterm} should roughly be $\frac{1}{m}$.  
\end{remark}
\begin{remark}\rm
Writing $m=\ell_1^{e_1}\ell_2^{e_2}...\ell_s^{e_s},$ for any $a\in\Z/m\Z$ we have an element in $\Z/\ell_1^{e_1}\Z\times\Z/\ell_2^{e_2}\Z\times...\times\Z/\ell_s^{e_S}\Z$ of the form $(u_1\ell_1^{n_1}, u_2\ell_2^{n_2},...,u_s\ell_s^{n_s})$ where $u_i$ are units in $\Z/\ell_i^{e_i}$, and $0\leq n_i\leq e_i.$ Then the number of ways of writing $a$ as product of $M$ elements in $(\Z/m\Z)$ (allowing repetition) is $$\phi(m)^{M-1}\prod\limits_{j=1}\limits^{s}\left(\sum\limits_{i_{M-2}=0}\limits^{n_j}\sum\limits_{i_{M-3}=0}\limits^{i_{M-2}}...\sum\limits_{i=0}\limits^{i_1}1\right).$$

Let us first look for the case $M=2$ and $m=\ell^e.$ Let us write $a=u\ell^n,$ where $u$ is a unit in $\Z/\ell^e\Z$ and $n\in\{0,1,2,...,e\}.$  Then the number of times that $a$ as a product of $M$ number of terms is the same as the number of times $\ell^n$ as a product of $M$ terms. Hence it is enough to compute the number of ways of writing $a$ as a product of two terms only for $\ell^n.$ Any $\ell^n$ can be written as the product of two terms in the following ways $$\ell^n=u(u^{-1}l^n)=(u\ell)(u^{-1}\ell^{n-1})...=(u\ell^{n-1})(u^{-1}\ell)=(u\ell^n)u^{-1},$$ 
for any unit $u\in \Z/\ell^e \Z.$ Hence the number of times $\ell^n$ can be written as product of two terms is  $(n+1)\phi(\ell^e)$. 

In general, the number of times $a$ can be written as product of $M$ terms is $$\left(\sum\limits_{i_{M-2}=0}\limits^{n}\sum\limits_{i_{M-3}=0}\limits^{i_{M-2}}\cdots \sum\limits_{i=0}\limits^{i_1}1\right)\phi(\ell^e)^{M-1},$$
This can be realized by noting that a product of $M$ terms is also a product of two terms; one is a product of $M-1$ terms. Then for any $a\in\Z/m\Z$ we have an element in $\Z/\ell_1^{e_1}\Z\times\Z/\ell_2^{e_2}\Z\times...\times\Z/\ell_s^{e_S}\Z$ under the natural isomorphism say $(u_1\ell_1^{n_1}, u_2\ell_2^{n_2},...,u_s\ell_s^{n_s})\in\Z/\ell_1^{e_1}\Z\times\Z/\ell_2^{e_2}\Z\times...\times\Z/\ell_s^{e_S}\Z,$ where $u_i$ are units and $0\leq n_i\leq e_i.$ Counting the number of ways of writing $a$ as a product of $M$ terms is equal to the product of the number of ways of writing each coordinate of $a$ as a product of $M$ terms.

\end{remark}

\section{On the residue classes }\label{sec:wring}
Let $f(z)$ be a cuspform with coefficients in $\mathbb{Q},$ and $m$ be any integer. In this section, we give a lower bound for the number of elements in the set $\{a(n) \pmod m\}_{n\in I},$ where $I$ is some small set. We know from section~\ref{sec:basicgalois} that the set is $\Z/m\Z,$ when $I$ is a large set, and $f(z)$ is of a certain type. In this section, we shall consider a small set $I.$ Shparlinski in \cite{Shparlinski2005} considered this for the Ramanujan-tau function. Arguing along the same lines, we first have the following generalization.

\begin{lemma}\label{lem:sizehecke} Let $f(z)$ be any Hecke eigenform, and $m$ be any integer. For any set of primes $S,$ consider 
$$N_{f,m,S}(x)= \#\{a(p^i)\pmod {m} \mid p\in S,~p\leq \sqrt{x},~i=1,2\}.$$ If $S$ has a positive density, then for any $x\geq 1$, 
$$N_{f,m,S}(x)\gg_{S}x^{1/4+o(1)},$$
provided that $x^{1/2}\leq L,$ where $L$ is the largest prime factor of $m.$ More precisely, $$\#\{a(p)\pmod {m} \mid p\in S,~p\leq \sqrt{x}\}\gg_{S}x^{1/4+o(1)}$$ 
 or
 $$\#\{a(p^2)\pmod {m} \mid p\in S,~p\leq \sqrt{x}\}\gg_{S}x^{1/4+o(1)}.$$ In particular $N_{f,L,S}(m^{2\varepsilon})\gg_{S} m^{\frac{\varepsilon}{2}+o(1)},$
for any $0<\varepsilon<1$, provided that $m^{\varepsilon}\leq L.$
\end{lemma}
The proof is essentially the same as in \cite{Shparlinski2005}. It follows from the Hecke relation $a(p^2)=a(p)^2-p^{k-1},$ and the fact that the number of distinct residue classes $p^{k-1} \pmod m,~p\leq \sqrt{x}\leq L$ is $\gg \sqrt{x}$. Given any integer $m,$ the condition $m^{\varepsilon}\leq L$ is, of course, satisfied for any small $\varepsilon>0.$ However if we want to take any $1/2\leq \varepsilon< 1,$ we should have that $\nu_L(m)=1$ and $L$ is sufficiently larger than the other prime factors of $m$.

Let $f_1,f_2,\cdots, f_r$ be a set of eigenforms of the same weight $k$ and level $N$, and consider 
\begin{equation}\label{eqn:setp}
\mathcal{S}_{f_1,f_2,\cdots,f_r,m}=\left\{p \mid a_1(p)\equiv a_2(p)\cdots\equiv a_r(p)\pmod{m},~p^{k-1}\equiv 1 \pmod m\right\}.
\end{equation}
Then we have the following.
\begin{lemma}
If all of the $f_1,f_2,\cdots,f_r$ are newforms without CM, then $\mathcal{S}_{f_1,f_2,\cdots,f_r,m}$ has a positive density of primes, if it is non-empty. Otherwise there exists an integer $N_f,$ such that for any integer $m$ co-prime to $N_f,$
$$a_i(p)\equiv \pm a_j(p)\pmod m,~p^{k-1}\equiv 1\pmod m~\forall~1\leq i,j\leq r,$$
for a set of primes $p$ with positive density.
\end{lemma}

\begin{proof}
Let us first start with recalling the Galois representation from (\ref{eqn:map}) $$\rho_{f_1,f_2,\cdots,f_r,m}:\mathrm{Gal}(\overline{\Q}/\Q)\to \mathrm{GL}_{2r}(\Z/m\Z).$$
Now consider 
$$C=\left\{\begin{psmallmatrix}
    A_1 & & \\
    & A_2 & &\\
    & & \ddots &\\
    & & & A_r\\ 
  \end{psmallmatrix}\in \mathrm{SL}_{2}(\Z/m\Z)^r~\big|~ \mathrm{tr}(A_1)=\mathrm{tr}(A_2)=\cdots=\mathrm{tr}(A_r)\right\}.$$
If $\mathcal{S}_{f_1,f_2,\cdots,f_r,m}$ is non-empty, then $C\cap \mathrm{im}(\rho_{f_1,f_2,\cdots,f_r,m})$ is also non-empty, and we have the required positive density due to Chebotarev's density theorem. 

On the other hand if $\mathcal{S}_{f_1,f_2,\cdots,f_r,m}$ is empty, then it follows from (\ref{eqn:factors}) that $\mr{G}_{f_1,f_2,\cdots,f_r,m}$ does not contain $\text{SL}_2(\Z/m\Z)^r.$ Then Lemma~\ref{lem:explimage} implies that there is more than one equivalence class in the set $\{f_1,f_2,\cdots,f_r\}.$ Let $f_{i_1},f_{i_2},\cdots,f_{i_{r'}}$ be the representatives from each class. Again applying Lemma~\ref{lem:explimage}, we see that $\mr{G}_{f_{i_1},f_{i_2},\cdots,f_{i_{r'}},m}$ contains $\text{SL}_2(\Z/m\Z)^{r'}.$ The proof is now complete due to (\ref{eqn:factors}). Also, note that the condition $p^{k-1}\equiv 1 \pmod m$ is satisfied because we are working with the conjugacy classes in $\text{SL}_2(\Z/m\Z).$
\end{proof}

Let us now consider $f_1,f_2,...,f_r$ be the Hecke eigenforms with Fourier expansion $f_i(z)=\sum\limits_{n=1}\limits^{\infty}a_i(n)z^n,$ and $a_i(n)\in \mathbb{Q},~\forall 1\leq i\leq r.$ For any homogeneous polynomial $P(x_1,x_2,...,x_r)$ with $P(\pm1,\pm1,...,\pm1)\neq 0,$ set $a(n)=P(a_1(n),a_2(n),...,a_r(n)).$ Consider the quantity $N_{f,m,S}(x)$ as in Lemma~\ref{lem:sizehecke}. Since we are assuming that $$P(\pm1,\pm1,...,\pm1)\neq 0,$$ we can also assume that $P(\pm1,\pm1,...,\pm1)\not\equiv 0 \pmod m$, for any integer $m$ with sufficiently large prime factors. 

\begin{corollary}\label{cor:lbd}
Suppose that $f_1,f_2,\cdots, f_r$ are all newforms without CM. Let $m$ be any integer with sufficiently large prime factors satisfying that $P(\pm1,\pm1,...,\pm1)\not\equiv 0 \pmod m$, and $L$ be the largest prime factor of $m$ satisfying that $m^{\varepsilon}\leq L$ for some $0<\varepsilon<1.$ Then we have,
$$N_{f,L,S}(m^{2\varepsilon})\gg \frac{m^{\varepsilon/2+o(1)}}{d},~d=\mathrm{deg}(P).$$

\end{corollary}

\begin{proof}
If the set $\mathcal{S}_{f_1,f_2,\cdots,f_r,m}$ in (\ref{eqn:setp}) is non-empty, then it follows from Chebotarev's density theorem that for any $n\geq 1$, 
$$a(p^n)\equiv P(a_1(p^n),a_1(p^n),\cdots a_1(p^n))\pmod {m},$$
for a set of primes $p$ with positive density. Since $P(x_1,x_2,\cdots x_r)$ is a homogeneous polynomial, we have
$$P(a_1(p^n),a_1(p^n), \cdots a_1(p^n))\equiv P(1,1, \cdots ,1)a_1(p^n)^d\pmod m,$$
where $d$ is the degree of $P.$ On the other hand for any prime $p$ not in $\mathcal{S}_{f_1,f_2,\cdots,f_r,m},$  
$$a(p^n)\equiv Q(a_1(p^n),a_1(p^n),\cdots a_1(p^n)) \pmod m,~\forall n\geq 1$$
for a set of primes $p$ with positive density, where $Q(x_1,x_2,...,x_r)\equiv P(\pm x_1,\pm x_2,\cdots,\pm x_r)$. Since $P(x_1,x_2,..,x_r)$ is a homogeneous, $Q(x_1,x_2,...,x_r)$ is a homogeneous polynomial of degree $d=\deg(P)$. Hence we get,
$$Q(a_1(p^n),a_1(p^n),...a_1(p^n))\equiv P(\pm 1,\pm 1,\cdots ,\pm 1)a_1(p^n)^d\pmod m.$$
The proof now follows immediately from Lemma~\ref{lem:sizehecke} and by the assumption that $P(\pm 1,\pm 1,\cdots ,\pm 1)\not\equiv 0.$ Note that the factor $1/d$ is coming because an equation $x^d\equiv a \pmod L$ has at most $d$ roots over $\mathbb{F}_L$.
\end{proof}
\begin{remark}\rm
Here we are always concerned with when all the $c_i$ are in $\mathbb{Q}.$ The number of tuples $(c_1,c_2,\cdots, c_r)$ of height at most $H$ is $\sim$ $(2H/\zeta(2))^r$, see \cite[Theorem B.6.2]{hindry2013diophantine} Among them, the number of tuples $(c_1,c_2,\cdots, c_r)$ with $\sum_{i=1}^{m} \pm c_i=0$ is $\sim H^{r-1}.$ In the sense of heights, we are saying that almost any $f$ in $S_k(\mathbb{Q}, N)$ of the form $\sum\limits_{i=1}\limits^{r} c_if_i,~c_i\in \Q$, satisfy the condition in Corollary~\ref{cor:lbd}.
\end{remark}

Now to study the case when not all of the $f_i$ are newforms without CM, we need to count the number of points on the intersection of certain hypersurfaces. In this case, we have a weaker lower bound in the sense of a lesser exponent.
 \begin{lemma}\label{lem:explicitcoeffs}
 Let us consider $a(n)=\sum\limits_{i=1}\limits^{r}a_i(n),$ then for any integer $m\geq 1,$ the sum $\sum\limits_{i=1}\limits^{r}a_i(p)^m$ can be written as a linear combination of $a(p^{m'}),~0\leq m'\leq m,$ where the coefficients are polynomials in $p$ with coefficients in $\mathbb{Q}$. Moreover, the coefficient associated with $1$ (resp. $a(p)$) has the highest degree when $m$ is odd (resp. even).
 \end{lemma}
\begin{proof}
By the properties of the Hecke operators, we have
\begin{align*}
a_i(p^{2\beta})&=a_i(p)^{2\beta}-\binom{2\beta-1}{1}p^{12}a_i(p)^{2\beta-2}+\dots+(-1)^{\beta-1}\binom{\beta+1}{2}p^{12\beta-12}a_i(p)^2+\\
& \;\, +(-1)^\beta p^{(k-1)\beta},
\end{align*}
    and
\begin{align*}
a_i(p^{2\beta+1})&=a_i(p)^{2\beta+1}-\binom{2\beta}{1}p^{12}a_i(p)^{2\beta-1}+\dots+(-1)^{\beta-1}\binom{\beta+2}{3}p^{12\beta-12}a_i(p)^3+\\
& \;\, + (-1)^\beta\binom{\beta+1}{1}p^{(k-1)\beta}a_i(p),
\end{align*}
for any $\beta\in \mathbb{N}$ and $i\in \{1,2,\cdots, r\}.$ Denoting $\sum\limits_{i=1}\limits^{r}a_i(p)^m=A_m,$ we see that $a(p^n)$ can be written as a linear combination of $A_1,A_2,\cdots, A_n,$ with the coefficients being polynomials in $p.$ The proof now follows inductively, as $A_1=a(p).$


 
\end{proof} 
 \begin{lemma}\label{lem:qtos}
 Let $a_1,a_2,...,a_r$ be any $r$ real numbers, $s_j=\sum\limits_{i=1}\limits^{r}a_i^j,$ and consider $f(x)=x^r+q_1x^{r-1}+q_2x^{r-2}+...+q_{r-1}x+q_r$ be the polynomial whose roots are $a_1,a_2,a_3,...,a_r.$ Then every coefficient $q_k$ can be written as a polynomial in $\{s_j\}_{ j\in\{1,2,...,r\} }.$
 \end{lemma}
 \begin{proof}
 The proof follows immediately from Newton's identity of the symmetric polynomials. More precisely, we have
 $$q_k=\frac{(-1)^k}{k!}B_k(-s_1,-1!s_2,\cdots,-(k-1)!s_k),$$
 for some polynomial $B_k\in \Q[x_1,x_2,\cdots, x_k].$

 \end{proof}

\begin{proposition}\label{prop:lowerbnd}
Let $f\in S(\Q,N)$ be any arbitrary element of the form $c_1f_1+c_2f_2+\cdots c_rf_r$, where $c_i\in \Q,$ and all the $f_i$ are Hecke eigenforms. For any set of primes $S$, and any integer $m,$ let us consider
$$N_{f,m,S}(x)= \#\{a(p^i)\pmod {m} \mid p\in S,~p\leq x^{\frac{1}{2r}},~i=1,2,3,...,2r\}.$$
Suppose that $S$ has positive density. Then for any $\delta>0$, and any sufficiently large $x\geq 1$, we have
$$N_{f,m,S}(x)\gg_{S} x^{\frac{1}{4r^2}-\delta},$$
provided that $x^{1/2r}\leq L,$ where $L$ is the largest prime factor of $m.$ In particular $N_{f,L,S}(m^{2\varepsilon})\gg_{S} m^{\frac{\varepsilon}{2r^2}-\delta},$
for any $\varepsilon>0,$ provided that $m^{\varepsilon/r^2}\leq L.$
\end{proposition}

\begin{proof}

Let $y>0$ be any given real number for which $N_{f,m,S}(x)\leq y.$ In particular, $\#\{a(p)\pmod{m}: p\leq x^{\frac{1}{2r}}\}<y,$ and hence there exists $a_1$ such that $a(p)\equiv a_1 \pmod{m}$ for at least $\frac{x^{\frac{1}{2r}+o(1)}}{y}$ primes up-to $x^{\frac{1}{2r}}.$ Now consider the set
 $$S_{a_1}=\{p: a(p)\equiv a_1 \pmod{m},~p\leq x^{\frac{1}{2r}}\}. $$ 
We have that $\#\{ a(p^2)\pmod{m}: p\in S_{a_1}(x)\}<y,$ then there exists $a_2$ such that, $a(p^2)\equiv a_2\pmod{m}$ for at least $\frac{\# S_{a_1}(x)}{y}=\frac{x^{\frac{1}{2r}+o(1)}}{y^2}$ many primes up-to $x^{\frac{1}{2r}}.$ Let us then consider
 $$S_{a_1,a_2}(x)=\{p: a(p^2)\equiv a_2 \pmod{m},~p\in S_{a_1}(x)\}.$$ Since $\{ a(p^3)\pmod{m}: p\in S_{a_1,a_2}(x)\}<y,$ there exists $a_3$ such that $a(p^3)\equiv a_3\pmod{m}$ for at least $\frac{\#S_{a_1,a_2}(x)}{y}=\frac{x^{\frac{1}{2r}+o(1)}}{y^3}$ many primes up-to $x^{\frac{1}{2r}}.$ Arguing recursively, we obtain $$a(p)=a_1\pmod m,~a(p^2)=a_2 \pmod m,\cdots, a(p^{2r})=a_{2r}\pmod m,$$ 
 for at least $\frac{x^{\frac{1}{2r}+o(1)}}{y^{2r}}$ many primes up-to $x^{\frac{1}{2r}},$ and we denote this set of primes to be $S_{a_1,a_2,...a_{2r}}(x).$

 Now the characteristic polynomial of the sequence $\{a_i(p^n)\}$ is $x^2-a_i(p)x+p^{k-1},$ and hence $\{a(p^n)\}$ is a linear recurrence sequence with the characteristic polynomial $p(T)=\prod\limits_{i=1}\limits^{r}(T^2-a_i(p)T+p^{k-1}).$\footnote{For a reference, the reader may look at \href{https://math.stackexchange.com/questions/1348838/sum-and-product-of-linear-recurrences}{https://math.stackexchange.com/questions/1348838/sum-and-product-of-linear-recurrences.}} It follows from Lemma~\ref{lem:explicitcoeffs} and Lemma~\ref{lem:qtos} that, the coefficients of $p(x)$ are polynomials in the prime $p,$ for any $p\in S_{a_1,a_2,...a_{2r}}(x).$ Moreover, the polynomial with the highest degree appears only once, with degree $(k-1)r.$ In particular, we get a polynomial $g(T)$ of degree $(k-1)r$, which satisfies
 $$g(p)\equiv 0\pmod m,~\forall p\in S_{a_1,a_2,...a_{2r}}(x).$$
In particular, $g(p)\equiv0\pmod L$ for at least $\frac{x^{\frac{1}{2r}+o(1)}}{y^{2r}}$ many primes $p$ up-to $x^{\frac{1}{2r}}\leq L$. The proof now follows, taking $y=x^{\frac{1}{4r^2}-\delta}$ since $g(p)\equiv 0\pmod L$ for only $O_g(1)$ many $p\leq L$.
\end{proof}

\section{Waring-type for problems with modular forms}\label{sec:wring1}
The main goal of this section is to extend the main result of Shparlinski for a large class of cusp forms and produce many other $n_i$s for the solutions. In certain cases, we also study the same problem modulo composite numbers. Let us first discuss the tools that will be used throughout.
\subsection{Exponential sums and Waring's type problems.}
Let $m,s,\omega\geq 1$ be any given integers, and $A_1,A_2,\cdots, A_{\omega}$ be some subsets of $\mathbb{Z}/m\mathbb{Z}$ satisfying
\begin{equation}\label{eqn:product}
\prod_{i=1}^{\omega} \# A_i\gg m^{1+\beta},
\end{equation}
for some $\beta>0.$ For any $a\in \Z/m\Z,$ we denote $T_{s}(a)$ be the number of solutions to the equation 
\begin{equation*}\label{eqn:sum}
\prod_{i=1}^{\omega}a_1^{(i)}+\prod_{i=1}^{\omega}a_2^{(i)}+...+\prod_{i=1}^{\omega}a_s^{(i)}\equiv a \pmod m,
\end{equation*}
where $a_j^{(i)}\in A_i,~\forall 1\leq j\leq s,1\leq i\leq \omega.$ We then have the following counting formula,
\begin{equation}\label{eqn:formula}
T_s(a)=\frac{(\#A_1\#A_2\cdots \#A_{\omega})^s}{m}+O\left(\frac{1}{m} \sum_{\lambda=1}^{m-1}\left|\sum_{a^{(1)}\in A_1, \cdots, a^{(\omega)}\in A_{\omega}} \en{\lambda  a^{(1)}a^{(2)}\cdots a^{(\omega)}}\right|^s  \right).
\end{equation}
When $\omega=2,$ an old result mentioned in Exercise 14.a in Chapter 6 of \cite{Vin54} says that by replacing in the proof of Lemma 7 in \cite{BG11}, $q$ by $m$ and taking $a:=1_{A_1},b:=1_{A_2}$ and $\phi:=\mathrm{e}_m,$ we have,
\begin{equation}\label{eqn:exp1} \mathrm{max}_{\lambda\in \Z/m\Z}\left|\sum_{a_1\in A_1}\sum_{a_2\in A_2}\en{\lambda a_1a_2}\right|\leq \sqrt{m\#A_1\#A_2}.
\end{equation}
Substituting (\ref{eqn:exp1}) in (\ref{eqn:formula}), we get the following result.
\begin{corollary}\label{lem:tool}
For $\omega=2$ and any $s>2/\beta,$ the sum
$$a^{(1)}_1a^{(2)}_1+a^{(1)}_2a^{(2)}_2+\cdots + a^{(1)}_sa^{(2)}_s$$ 
is equidistributed in $\mathbb{Z}/m\mathbb{Z}$ for any sufficiently large $m$, where $\beta$ is the same constant as in (\ref{eqn:product}), and $a^{(1)}_j\in A_1, a^{(2)}_j\in A_2,\forall 1\leq j\leq s.$
\end{corollary}

When $\omega= 3$, we shall use the following bound by Shkredov in \cite[Theorem 5]{Shk18}.
\begin{theorem}\label{thm:Shkredov}
Let $A_1, A_2, A_3 \subseteq \mathbb{F}_{\ell}$ be arbitrary sets such that for some $\delta > 0$ the following holds 
\begin{equation}
|A_1||A_2 ||A_3| \geq \ell^{1+\beta}.
\end{equation}
Then
$$
 \mathrm{max}_{\lambda\in (\Z/\ell\Z)^{*}}\left|\sum_{a_1\in A_1} \sum_{a_2\in A_2}\sum_{a_3\in A_3}\el{\lambda a_1a_2a_3}\right| \ll \frac{|A_1||A_2||A_3|}{\ell^{\frac{\beta}{
8 \log(8/\beta)+4}}}.$$
\end{theorem}

To treat the case $\omega> 3$, the following bound due to Bourgain, Gilbichuck in \cite[Theorem 2]{BG11} will be handy for us. Even the following bound is true for $\omega=3,$ we use this only for the $\omega>3$ case. 

\begin{theorem}[Bourgain-Glibichuk]\label{thm:exp2}
Let $3 \leq \omega \ll \log \log \ell$ be a natural number and $\varepsilon > 0$ an arbitrary fixed constant. For any subsets $A_1, A_2,\cdots, A_{\omega} \subset \mathbb{F}_{\ell} \setminus \{0\}$ with
\begin{equation}\label{eqn:cond}
|A_1| \cdot |A_2| \cdot (|A_3| \cdots |A_{\omega}|)^{1/81} > \ell^{1+\beta},
\end{equation}
there is an estimate
$$\mathrm{max}_{\lambda\in (\Z/\ell\Z)^{*}}\left|\sum_{a_1\in A_1}\sum_{a_2\in A_2}\cdots \sum_{a_{\omega}\in A_{\omega}}\el{\lambda a_1a_2\cdots a_{\omega}}\right|\ll \frac{|A_1||A_2|\cdots |A_{\omega}|}{\ell^{0.45\beta/2^{\omega}}}.$$
\end{theorem}

To study modulo composite numbers, we need to study these exponential sums over arbitrary finite fields $\mathbb{F}_q.$ For which, we could use Theorem 4 of Bourgain-Glibichuk in \cite{BG11}. With this, we get a non-trivial bound assuming that, for any $d\in \mathbb{F}_q$ and any proper subfield $S$ of $\mathbb{F}_q,~dS$ has a small intersection with each of the set $A_i.$ However in our case, each of the sets $A_i$ will be in a prime field $\mathbb{F}_{\ell},$ and hence, we can not use this result. However, Theorem~\ref{thm:exp2} could be used to study the square-free integers. To be more precise, using these two results, we have the following.

\begin{corollary}\label{cor:shk}
Let $m$ be any square-free integer, and $A_1, A_2,\cdots, A_{\omega} \subseteq (\Z/m\Z)^{*}$ with
\begin{equation}
|A_1| \cdot |A_2| \cdot (|A_3| \cdots |A_{\omega}|)^{1/81} > m^{1+\beta}.
\end{equation}
Then we have the following estimate
$$\mathrm{max}_{\lambda\in (\Z/m\Z)^{*}}\left|\sum_{a_1\in A_1}\sum_{a_2\in A_2}\cdots \sum_{a_{\omega}\in A_{\omega}}\en{\lambda a_1a_2\cdots a_{\omega}}\right|\ll \frac{|A_1||A_2|\cdots |A_{\omega}|}{\ell^{0.45\beta/2^{\omega}}},$$
for some prime factor $\ell$ of $m$.
\end{corollary}
\begin{proof}
    Note that there exists a prime $\ell \mid m$ for which 
    $$
|A^{(\ell)}_1| \cdot |A^{(\ell)}_2| \cdot (|A^{(\ell)}_3| \cdots |A^{(\ell)}_{\omega}|)^{1/81} > \ell^{1+\beta},
$$
where $A_i^{(\ell)}$ is denoted to be $\ell^{\text{th}}$ component of $A_i,~\forall 1\leq i\leq \omega$. The proof now follows applying Theorem~\ref{thm:exp2} for $\{A^{(\ell)}_i\}_{1\leq i\leq \omega}$, and trivially estimating the exponential sum associated to the other components. 
\end{proof}
\subsection{Sums with Hecke eigenforms}

Let $m$ be any given integer, and $f$ be any Hecke eigenform. We then want to show that $\{a(n) \pmod m\}_{n=m^{O(1)}}$ is an additive basis for $\Z/m\Z.$ This was proved by Shparlinski when $f$ is given by the Ramanujan-tau function and $m$ is a prime. For any $\gamma>0,$ let us consider
$$N_{\gamma}=\{m \in \mathbb{N}\mid \ell~\mathrm{prime~divides}~m\implies m \leq \ell^{1+\gamma}\}.$$ 

\begin{proposition}\label{prop:twop}
Let $\omega$ be any integer, and $\gamma, \beta>0$ be any real numbers. Take any pairwise disjoint set of primes $S_1,S_2,\cdots,S_{\omega}$ satisfying 
\begin{equation}\label{eqn:beta}
\begin{cases}
    \#A_1\#A_2\cdots \#A_{\omega}\geq m^{1+\beta},~\omega=2 ,3\\
    \#A_1\#A_2(\#A_3\cdots \#A_{\omega})^{\frac{1}{81}}\geq m^{1+\beta},~\omega\geq4.
\end{cases}
\end{equation}
where $A_i=\{a(p), a(p^2) \pmod {m}\mid p\in S_i\}.$ 
 Set $B_{\gamma}(\ell)=\{m\in N_{\gamma}~\text{is square-free and}~p\mid m\implies p>\ell\}.$ There exists $\ell_\beta$ such that
\begin{equation*}
\sum_{i=1}^{s}a(n_i)\equiv a\pmod{m}
\end{equation*} is solvable for any integer $a,$ and \begin{equation}
  s>  \begin{cases}
        \frac{2}{\beta} ,\  \omega=2 \\
        \frac{(1+\gamma)(8\log(\frac{8}{\beta})+4)}{\beta} ,\  \omega=3,  & m\in B_{\gamma}(\ell_\beta)\\
        \frac{(1+\gamma)2^\omega}{0.45\beta}, \ \omega\geq 4, & m\in B_{\gamma}(\ell_\beta).
        
    \end{cases}
\end{equation}
In either of the cases, any such $n_i$ has prime factors only from $S_1,S_2,\cdots,S_{\omega}.$
\end{proposition}
\begin{proof}
 The proof for the case $\omega=2$ follows from Corollary~\ref{lem:tool}, since $a(\cdot)$ is multiplicative and $S_1,S_2$ are disjoint set of primes.\\

For higher values of $\omega$, we assume that $m$ is squarefree. To prove for $\omega=3,$ note that there exists a prime $\ell\mid m$ such that 
\begin{equation*}
|A^{(\ell)}_1||A^{(\ell)}_2|\cdots |A^{(\ell)}_{3}|\geq \ell^{1+\beta},
\end{equation*}
where $A^{(\ell)}_i$ is the $\ell^{\text{th}}$ component of $A_i.$ Arguing similarly as in the proof of Corollary~\ref{cor:shk}, we get the following from Theorem~\ref{thm:Shkredov}.
$$
 \mathrm{max}_{\lambda\in (\Z/m\Z)^{*}}\left|\sum_{a_1\in A_1} \sum_{a_2\in A_2}\sum_{a_3\in A_3}\en{\lambda a_1a_2a_3}\right| \ll \frac{|A_1||A_2||A_3|}{\ell^{\frac{\beta}{
8 \log(8/\beta)+4}}}.$$
Now note that there exists $\ell_\beta$ such that, the following holds for any $m\in B(\ell_{\beta})$,
$$O\left(\left(\frac{|A_1||A_2||A_3|}{\ell^{\frac{\beta}{
8 \log(8/\beta)+4}}}\right)^s\right)=o\left(  \frac{(|A_1||A_2||A_3|)^s}{m}\right),$$
since $s\beta> (8\log(8/\beta)+4)(1+\gamma)$ by the assumption. The result follows this case from the formula at (\ref{eqn:formula}).

For a proof of $\omega\geq 4$, we follow the same argument as is the previous case and use Corollary~\ref{cor:shk} and (\ref{eqn:formula}).

\end{proof}

\subsection{Sums with a larger class}
Let us now consider a modular form $f$ with rational coefficients of the form $c_1f_1+c_2f_2+\cdots+ c_rf_r,$ where $c_i\in Q,$ and $f_i$ are all Hecke eigenforms with rational coefficients. More generally, one can also consider a new sequence $a(n):=P(a_1(n),a_2(n),\cdots, a_r(n))$ for any homogeneous polynomial $P(x_1,x_2,\cdots,x_r)$ with rational coefficients. The first problem we immediately encounter is that $a(\cdot)$ is not necessarily multiplicative unless $P(\cdot)$ is a monomial. Even in the case of a monomial, to get an analogous result to Proposition~\ref{prop:twop}, we need to ensure that 
\begin{equation*}
\#A_1\#A_2\cdots \#A_{\omega}\geq m^{1+\beta},
\end{equation*}
where $$A_i=\left\{P(a_1(p),\cdots,a_r(p)),~P(a_1(p^2),\cdots, a_r(p^2)) \pmod {m}\mid p\in S_i\right\},$$
for some set of primes $S_i.$ This is easy if $P$ is of the form $x_i^e$ for some $i.$ In general, we have a somewhat weaker result, which shall be discussed in this section. For any $r$-tuple of signs $\vec{\sigma},$ let us consider
$$\mathcal{S}_{\vec{\sigma},\vec{f},m}=\left\{p \mid \sigma_1a_1(p)=\cdots=\sigma_ra_r(p) \pmod m,~p^{k-1}=1\pmod m\right\},$$
and $\mathcal{S}_{\vec{\mathrm{sign}},\vec{f},m}=\bigcup_{\vec{\sigma}\in \{\pm 1\}^r}\mathcal{S}_{\vec{\sigma},\vec{f},m}.$ We then have the following.
\begin{corollary}\label{cor:cuspcase}
Let $m,\omega,\gamma$ and $\beta>0$ be as in Proposition~\ref{prop:twop}. Take any pairwise disjoint set of primes $S_1, S_2,\cdots, S_{\omega}$ satisfying 
\begin{equation}
\begin{cases}
    \#A_1\#A_2\cdots \#A_{\omega}\geq m^{1+\beta}, \omega=2 ,3\\
    \#A_1\#A_2(\#A_3\cdots \#A_{\omega})^{\frac{1}{81}}\geq m^{1+\beta}, \omega\geq4.
\end{cases}
\end{equation}
where $A_i=\{a(p), a(p^2) \pmod {m}\mid p\in S_i\cap\mathcal{S}_{\vec{\mathrm{sign}},\vec{f},m}\}.$ Suppose that $(P(\vec{\sigma}),m)=1$ for any $\vec{\sigma}\in \{\pm 1\}^r.$ Then there exists $\beta'>0$ (depending on $\beta$) such that for $\omega=2,$ $3$ and $\omega\geq 4$ respectively, and for $s>\frac{2}{\beta'},$ $\frac{(1+\gamma)8\log(8/\beta)+4}{\beta'},$ and $\frac{(1+\gamma)2^\omega}{0.45\beta'},$ any $a\in\mathbb{Z}$ can be written as
$$a(n_1)+a(n_2)+...+a(n_s) \equiv a\pmod{m}, \ n_i\in\mathbb{N}, \ i=1,2,\cdots,s,$$
for any sufficiently large $m$. Moreover any such $n_i$ has prime factors only from $S_1,S_2,\cdots,S_{\omega}.$
\end{corollary}
\begin{proof}
Note that for any $p\in \mathcal{S}_{\vec{\mathrm{sign}},\vec{f},m},$ we have
$$a_i(p^2)=a_j(p^2),~\forall~1\leq i,j\leq r.$$
In particular, for any such prime $p$, 
$$a(p)=P(\vec{\sigma})a_1(p)^d,~a(p^2)=P(1,1,\cdots, 1)a_1(p^2)^d,$$
for some $\vec{\sigma}\in \{\pm 1\}^{r}$ and $d$ is the degree of $P.$ Since $\#A_1\#A_2\cdots \#A_{\omega}\geq m^{1+\beta},$ for a particular type of sign-tuple $\vec{\sigma}:=(\sigma_1,\sigma_2,\cdots, \sigma_r)$, we have
$$\#A'_1\#A'_2\cdots \#A'_{\omega}\geq \frac{m^{1+\beta}}{2^r}, \ \omega=2,3$$
$$\#A'_1\#A'_2(\#A'_3\cdots \#A'_{\omega})^{\frac{1}{81}}\geq \frac{m^{1+\beta}}{2^r}, \ \omega\geq 4,$$
where $A'_i=\{a(p),~a(p^2) \pmod {m}\mid p\in S_i\cap \mathcal{S}_{\vec{\sigma},\vec{f},m}\}.$ The proof now follows from Proposition~\ref{prop:twop}, for any $\beta'$ and sufficiently large $m$ satisfying $m^{\beta-\beta'}\geq 2^r$.
\end{proof}
\subsection{Proof of the main results}
To prove Theorem~\ref{thm:main}, we need an explicit value of $\beta$ in (\ref{eqn:beta}). We shall obtain this by the known explicit bounds for the sum-product problems over finite fields. For instance, suppose that $m:=\ell$ is a prime and $A\subset \mathbb{F}_{\ell}$ is a small set. Then the problem is to find $\beta>0$ for which 
$$ \mathrm{max}\{|A+A|,|A\cdot A|\}\gg |A|^{1+\beta}.$$
Garaev in \cite{Garaev07} showed that $\beta$ could be taken to be $1/14,$ then Rudnev in \cite{Rudnev12} improved it to $1/11$ and the most optimal $\beta,$ according to the best of our knowledge is given by $1/5.$ This is a result of Roche-Newton, Shkredov, and Rudnev in \cite{Rudnev16}.

\subsection{Proof of Theorem~\ref{thm:main}} 
For the proof of part $(a)$, note that $\#\{a(p)\pmod L,~p\in S_1,p\leq m^{\epsilon}\}~\mathrm{or}~\#\{a(p^2)\pmod L,~p\in S_1,p\leq m^{\epsilon}\}\gg_{S_1} m^{\frac{\epsilon}{2}+o(1)}\geq L^{\frac{\epsilon}{2}+o(1)}.$ Consider the set with a larger size and set it as $A_1.$ Similarly, $\#\{a(p)\pmod L,~p\in S_2,p\leq m^{\epsilon}\}~\mathrm{or}~\#\{a(p^2)\pmod L,~p\in S_2,p\leq m^{\epsilon}\}\gg_{S_2} m^{\frac{\epsilon}{2}+o(1)}\geq L^{\frac{\epsilon}{2}+o(1)}$, and denote the larger one as $A_2.$ Now we use \cite[Theorem 6]{Rudnev16} to both of the sets $A_1$ and $A_2$. We have set $A'_1$, which is one of the $A_1\cdot A_1$ or $A_1+A_1$, and a set $A'_2$, which is one of the $A_2\cdot A_2$ or $A_2+A_2$, satisfying that
$$\#A'_1\gg_{S_1} L^{\frac{\epsilon(1+\beta)}{2}+o(1)},~A'_2\gg_{S_2} L^{\frac{\epsilon(1+\beta)}{2}+o(1)},$$
when $L$ is sufficiently large say  $L^{\gamma}\geq L'$. Where $L'=\frac{m}{L}$ with $(L,L')=1.$ We have $$\#A'_1.~\#A'_2 \gg_{S_1,S_2} L^{\epsilon(1+\beta)+o(1)}=L^{(1+\gamma)(1+\beta')}\geq (LL')^{1+\beta'}=m^{1+\beta'+o(1)},$$ where $\beta'=\frac{\epsilon(1+\beta)}{1+\gamma}-1+o(1)$. On the other hand, realizing $A'_1$ and $A'_2$ as subsets of $\Z/m\Z$ under the natural inclusion $\Z/L\Z \hookrightarrow \Z/m\Z$, the proof now follows from Corollary~\ref{lem:tool} for any $s=4\left( \lfloor\frac{2}{\beta'}\rfloor +1\right)$. This is because any element in $A'_1\cdot A'_2$ is of the form $a(n_1)+a(n_2)+a(n_3)+a(n_4)$, with $n_1,n_2,n_3,n_4\leq m^{4\epsilon}$, each $n_i$ has one prime factor from $S_1$, and one prime factor from $S_2$. Furthermore, note that $s$ is indeed a computable constant depending on $\varepsilon$ and $\gamma$, as the reader can any plug any known values of $\beta$, mentioned in the preamble of this section.

Now, for a proof of part $(b),$ it follows from Lemma~\ref{lem:explimage} that all any two $f_i,f_j$ differ by a quadratic character. In particular, the result now follows the same argument as in the proof of Corollary~\ref{cor:cuspcase} by taking $d=1$ and $P=\sum c_ix_i$.\\

\begin{remark}\rm
Substituting $\beta=\frac{1}{5}, \ \epsilon=\frac{65}{66}$ and, $\gamma=\frac{1}{77}$ we get $s=52,n_i\leq\frac{260}{66},\forall 1\leq i\leq 52$, and recover Corollary~\ref{thm:main0} from Theorem~\ref{thm:main}.
\end{remark}

\subsection{Proof of Theorem~\ref{thm:main2}}
Let us first prove $(a).$ Take $\varepsilon>0$ be any given real, and $S_1, S_2,\cdots, S_{\omega}$ be any pairwise disjoint set of primes of positive density, with $\varepsilon(2+\frac{\omega-2}{81})>2$. We studied the case $\omega=2$ in Theorem~\ref{thm:main}. The proof for $\omega>2$ case follows similarly. It follows from Lemma~\ref{lem:sizehecke} that for each $1\leq i\leq \omega,$ we have
$\#\{a(p)\pmod{L},p\in S_i,p\leq m^{\varepsilon}\}~\mathrm{or}~\#\{a(p^2)\pmod{L},p\in S_i,p\leq m^{\varepsilon}\}\gg_{S_i} L^{\varepsilon/2+o(1)}.$ Denote $A_i$ to be one of the corresponding sets with larger cardinality, we have
$$\#A_1\#A_2 (\prod_{3\leq i\leq \omega} \#A_i)^{1/81}\gg_{S_1,S_2,\cdots, S_{\omega}} L^{\varepsilon(2+\frac{\omega-2}{81})/2+o(1)}.$$
Denoting $\beta=\varepsilon(2+\frac{\omega-2}{81})/2+o(1)-1$ (which is positive by the assumption on $\omega$) and writing $m=LL'$ with $(L,L')=1,$ we have 
$$L^{1+\beta}\geq m^{1+\beta'},$$
for any $\beta'$ satisfying $L^{\frac{\beta-\beta'}{1+\beta'}}>m/L.$ The result now follows from Proposition~\ref{prop:twop} for $s>\frac{(1+\gamma)8\log(8/\beta)+4}{\beta'}, \ \omega=3$ and $s>\frac{2^{\omega}}{0.45\beta'}, \ \omega\geq4,$ where $\beta'=\frac{\beta-\gamma}{\gamma+1}$.\\
Proceeding similar to the proof of part (b) of the previous proof, we get the part (b) of this theorem.\\
\begin{remark}\rm
We now list the explicit values in the following table, obtained from Theorem~\ref{thm:main2}.
\begin{center}
		\begin{tabular}{ p{2cm} p{2cm}p{2.5cm}p{1.5cm}}
			\hline
$\omega$ & $\epsilon$ & $\gamma$ &
$s$ \\
\hline

$21$\ & $0.9$ & $0.005$ & $2^{31}$ \\
$165$ & $0.5$ & $0.003$ & $2^{180}$ \\
$1461$ & $0.1$ & $0.0006$ & $2^{1478}$ \\
$16041$ & $0.1$ & $0.00006$ & $2^{16062}$\\ 
$161841$ &  $0.001$  &  $0.000006$  & $2^{161866}$ \\
$1619841$  &  $0.0001$ & $0.0000006$  & $2^{1619894}$\\
$16199841$  &  $0.00001$  &  $0.00000006$  & $2^{16199872}$\\
$161999841$  &   $0.000001$ &  $0.000000006$ & $2^{161999875}$\\ 
\hline

\end{tabular}\\
		\vspace{0.4cm}
		Table 1: Required number of terms for a given bound
	\end{center}

\end{remark}

\section{Further questions and remarks}
\subsection{Solution with primes} We are having some assumptions on the composite number $m$ in both Theorem~\ref{eq:main1} and Theorem~\ref{thm:main2}. We would like to see if it is possible to remove them. We also ask if it is possible to obtain a solution to the equation
$$\sum_{i=1}^{O(1)} a(p_i)\equiv a \pmod m,~p_i\leq m^{O(1)},$$
where each $p_i$ is a prime. Or at least, if $\{a(n)\pmod m\}_{\substack{\omega(n)=1\\n\leq m^{O(1)}}}$ is an additive basis for $\Z/m\Z.$ Recall that Bajpai, Garc\'{\i}a, and the first author studied this in \cite{BBG22}; however, their method does not give polynomial growth of the solutions. Note that we have obtained solutions with polynomial growth and $\omega$ many prime factors for certain $m$.\\

\subsection{Sum of the polynomial values} Following the arguments in Section~\ref{sec:wring} one can study solvability of the equation $\sum\limits_{i=1}\limits^{O(1)} a(n_i)^{d}\equiv a~\pmod{m},$ as remarked by Shparlinski in \cite{Shparlinski2005}. However, Proposition~\ref{prop:lowerbnd} is giving the hope that it is also possible to study $\sum\limits_{i=1}\limits^{O(1)} p(a(n_i))\equiv a~\pmod{m},$ for any polynomial $p(x)\in \Q[x].$ The only obstacle is that $p(f(n))$ may not be multiplicative for any multiplicative function $f(n).$ We also ask if there is a way to overcome this. Perhaps the most interesting situation is when $P$ is of degree $1.$ In that case, $a(n):=P(a_1(n),a_2(n),\cdots, a_r(n))$ is Fourier coefficient of some modular form.\\

\subsection{On a larger family of cuspforms} We expect that it is possible to work with a larger class of cuspforms in Proposition~\ref{prop:lowerbnd}, at least when $m$ is a prime $\ell$. We covered some other families in Corollary~\ref{cor:lbd}. In this section, we shall discuss our heuristics for extending these families. This is because we expect the following to hold, under some suitable conditions, perhaps. 

\begin{question}\label{qn:soln}
Let $\ell$ be any prime, and $L$ be its any power. Is it true that for any tuples $(c_i)_{1\leq i\leq r}\in \mathbb{F}_{L}^{r},~(a_i)_{1\leq i\leq r+1}\in \mathbb{F}_{L}^{r+1},$ the number of solutions to the equations
$$ c_1x_1^i+  c_2x_2^i +\cdots c_{r}x_{r}^i=a_i,~\forall 1\leq i\leq r+1,$$
is at most $O_r(1)?$
\end{question}
For instance, this is easily seen to be true when all the $c_i$ are the same, using Newton's identity. In general, we have a partial answer due to the following.

\begin{lemma}\label{lem:dim}
Let $\ell$ be any prime, and $L$ be its any power. For any $(c_1,c_2,\cdots,c_{r})\in \mathbb{F}_{L}^{r},$ consider $f_i=c_1 x_1^i+c_2x^i\cdots c_{r}x_{r}^i \in \mathbb{F}_{L}[x_1,x_2,\cdots x_{r}],$ and $V$ be the projective variety generated by $f_1,f_2,\cdots, f_{r}$. Then $\mathrm{dim}(V)=0$ provided that $\sum\limits_{i\in S} c_i\neq 0$ in $\mathbb{F_\ell}$ for any $S\subseteq\{1,2,\cdots,r\}.$
\end{lemma}
\begin{proof}
It is enough to prove that there is no non-trivial prime ideal $I$ in the coordinate ring of $V.$ Suppose there is such a non-trivial prime ideal $I$. It follows from the identity 
$f_r=\sum_{i=1}^{r} q_i f_{r-i}$ that $(\sum_{i=1}^{r} c_i)x_1x_2\cdots x_r$ in $I$, where $q_i$ is $i^{\text{th}}$ elementary symmetric polynomial in $x_1,x_2,\cdots, x_r$. Since $\sum_{i=1}^{r} c_i \neq 0,$ we may assume that $x_{r}\in I.$ Then repeating the same argument, and keeping in mind that any sum $\sum_{i\in S} c_i\neq 0,$ we have that $x_1,x_2,\cdots,x_{r}$ all are in $I.$ In particular, this shows that, 
$$I=(f_1,f_2,\cdots,f_{r})=(x_1,x_2,\cdots,x_{r}),$$
which completes the proof.
\end{proof}
\begin{remark}\rm
The condition $\sum\limits_{i\in S} c_i\neq 0$ is important. For instance, $\sum\limits_{1\leq i\leq r} c_i=0,$ implies that $V$ contains the variety $(x_1-x_r,x_2-x_r,\cdots x_{r-1}-x_r).$ In particular, $V$ is of dimension at least $1.$ 
\end{remark}
As a consequence of Lemma~\ref{lem:dim}, we have a positive answer to Question~\ref{qn:soln} when all the $a_i$ are equal. For proof, one may use \cite[Theorem 2.1]{LR15}. We shall now see how helpful it is to have a complete answer to Question~\ref{qn:soln}. Let $f_1,f_2,\cdots,f_r$ be any set of Hecke eigenforms and set 
$$ a_{P,\vec{f}}(n)=P(a_1(n),a_2(n),\cdots, a_r(n)),$$
where $P(x_1,x_2,\cdots, x_r)$ is a polynomial with $r$ number of varriables. 

\begin{corollary}
Let $\ell$ be any given prime, then for any set of primes $S,$ consider quantity,
$$N_{S,P,\vec{f}}(x)=\#\{  a_{P,\vec{f}}(n) \pmod {\ell} \mid  p~\mathrm{divides}~n\implies p\in S,~n\leq x\}.$$
Suppose that $S$ has positive density, then we have the following estimate for any sufficiently large prime $\ell$, and any $\delta>0$
$$N_{S,P,\vec{f}}(x)\geq \mathrm{min}\left\{\ell^{\frac{1}{2}-\delta},x^{\frac{1}{4r^2}-\delta}\right\}.$$
\end{corollary}

\begin{proof}
Let us start with writing $a_i(p^n) \pmod{\ell}=c_i\alpha_i^n+d_i\beta_i^n,$ where $\alpha_i,\beta_i\in \mathbb{F}_{\ell^2}$, and suppose that $P$ is a homogeneous polynomial of degree $d$. Then $a_f(p^n)$ is a linear combination of 
$\left\{ \prod_{i=1}^{r}\alpha_i^{nt_i}\beta_i^{n(d_i-t_i)} \right\}_{\substack{0\leq t_i\leq d_i\\ \sum d_i=d}}.$ Let $\vec{d}=(d_1,d_2,\cdots, d_r)$ appear as degrees of a monomial in $P$. For a fixed tuple $(a_1,a_2,\cdots,a_{r+1}) \in \mathbb{F}_{\ell}^{r+1},$ let us now consider the number of primes $p$ for which $$(a_{P,\vec{f}}(p),a_{P,\vec{f}}(p^2) \cdots, a_{P,\vec{f}}(p^{r+1}))\pmod {\ell}=(a_1,a_2,\cdots, a_{r+1}).$$ It follows from our expectation in Question~\ref{qn:soln} that, $\left\{ \prod_{i=1}^{r}\alpha_i^{t_i}\beta_i^{d_i-t_i} \right\}_{\substack{0\leq t_i\leq d_i\\ \sum d_i=d}},$ is $O(1).$ In particular, 
$$\prod_{\substack{\sum d_i=d\\}} \prod_{i=1}^{r}\prod_{0\leq t_i\leq d_i}\alpha_i^{t_i}\beta_i^{d_i-t_i}=O(1).$$
Recall that $\alpha_i\beta_i=p^{k_i-1}\pmod{\ell}$, and hence
$$ \prod_{\substack{\sum d_i=d}} \prod_{i=1}^{r}\prod_{0\leq t_i\leq d_i}\alpha_i^{t_i}\beta_i^{d_i-t_i}=\prod_{\substack{\sum d_i=d}} \prod_{i=1}^{r-1}p^{(k_i-1)\frac{d_i(d_i+1)}{2}} \pmod{\ell}=O(1).$$
This is impossible since $S$ is infinite, and any $k_i-1$ is strictly positive. In particular, this shows that the number of primes $p$ for which $$(a_{P,\vec{f}}(p),a_{P,\vec{f}}(p^2) \cdots, a_{P,\vec{f}}(p^{r+1}))\pmod {\ell}=(a_1,a_2,\cdots, a_{r+1})$$
is $O(1).$ The proof now follows, arguing similarly as in the proof of Proposition~\ref{prop:lowerbnd}
\end{proof}

\section*{Acknowledgments}
We want to thank the Georg-August-Universit\"at G\"ottingen, Germany, for providing excellent working conditions. Most of the works for this project took place at IISER TVM, India, during the first author's visit. We are sincerely grateful for the hospitality and financial support of IISER. 

The first author is supported by ERC Consolidator grant
648329, and the fellowship of IISER TVM. The third author would like to thank CSIR for the financial support.

\end{document}